\documentclass[12pt,reqno]{amsart}
\usepackage{amsmath}
\usepackage{amssymb}
\usepackage{amstext}
\usepackage{a4wide}
\usepackage{graphicx}
\allowdisplaybreaks \numberwithin{equation}{section}
\usepackage{color}

\numberwithin{equation}{section}

\newtheorem{theorem}{Theorem}[section]

\newtheorem{lemma}[theorem]{Lemma}

\theoremstyle{definition}

\newtheorem{definition}[theorem]{Definition}

\theoremstyle{remark}
\newtheorem{remark}[theorem]{Remark}

\begin{document}

\title
[Steady vortex patches near a rotating flow]{Steady vortex patches near a rotating flow with constant vorticity in a planar bounded domain}

 \author{Guodong Wang, Bijun Zuo}

\address{Institute for Advanced Study in Mathematics, Harbin Institute of Technology, Harbin 150001, P. R. China}
\email{wangguodong14@mails.ucas.ac.cn}
\address{Institute of Applied Mathematics, Chinese Academy of Science, Beijing 100190, and University of Chinese Academy of Sciences, Beijing 100049,  P.R. China}
\email{bjzuo@amss.ac.cn}


\begin{abstract}
In this paper, we study steady vortex patch solutions to the incompressible Euler equations in a planar bounded domain $D$. Let $\psi_0$ be the solution of the elliptic problem $-\Delta \psi _{0} =1$ in $D$; $\psi_0=0$ on $\partial D$. We prove that for any finite collection of isolated maximum points of $\psi_0$, say $\{x_1,\cdot\cdot\cdot,x_k\},$ and any $k$-tuple $\vec{\kappa}=(\kappa_1,\cdot,\cdot,\cdot,\kappa_k)$ with $\kappa_i>0$ and $|\vec{\kappa}|:=\sum_{i=1}^k\kappa_i<<1,$ there exists a steady solution of the Euler equations such that the vorticity
has the form $\omega^{\vec{\kappa}}=1-I_{\cup_{i=1}^k A^{\vec{\kappa}}_i}$, where $I$ denotes the characteristic function, $|A^{\vec{\kappa}}_i|=\kappa_i$ and $A^{\vec{\kappa}}_i$ ``shrinks" to $x_i$ as $|\vec{\kappa}|\to 0$.


\end{abstract}

\maketitle

\section{Introduction}
In this paper, we are concerned with steady solutions of the two-dimensional incompressible Euler equations in a bounded domain $D$. The governing equations with prescribed boundary condition are as follows
\begin{equation}\label{100}
\begin{cases}
(\mathbf{v}\cdot\nabla)\mathbf{v}=-\nabla P,&x\in D,\\
 \nabla\cdot\mathbf{v}=0,&x\in D,\\
 \mathbf{v}\cdot\mathbf{n}=g,&x\in\partial D,
\end{cases}
\end{equation}
where $\mathbf{v}=(v_1,v_2)$ is the velocity field of the fluid, $P$ is the scalar pressure, $x=(x_1,x_2)$ is the space variable, $\mathbf{n}$ is the outward unit normal of $\partial D$, the boundary of $D$, and $g$ is the normal component of $\mathbf{v}$ on $\partial D$ satisfying the following compatibility condition
\[\int_{\partial D}g(x)dS_x=0.\]

The $curl$ of the velocity field, called the scalar vorticity, defined by $\omega:=\partial_{x_1}v_2-\partial_{x_2}v_1$, is an elementary quantity and plays a crucial role in the study of the Euler equations. See \cite{MajBer}\cite{Saff} for example. In terms of the vorticity, system \eqref{100} can be simplified as a single equation as follows
\begin{equation}\label{110}
\mathbf{v}\cdot\nabla\omega=0,\,\, \mathbf{v}=K\omega,
\end{equation}
where the relation $\mathbf{v}=K\omega$ is called the Biot-Savart law, and the operator $K$ is determined by the domain $D$ and the boundary condition $g$. As we will see in Section 2, for the case $D$ is smooth and simply-connected and $g\equiv0$, $K$ can be expressed in terms of the Green's function for $-\Delta$ in $D$ with zero Dirchlet data. Equation \eqref{110} is usually called the vorticity equation.
Mathematically, the vorticity equation is much easier to handle. Moreover, the velocity can be completely recovered from the vorticity via the Biot-Savart law. For these reasons, we will focus our attention on the vorticity equation in the rest of this paper.

Here we are mainly interested in the vortex patch solution of the Euler equations \eqref{100}, that is, the vorticity has the form $\omega=\lambda I_{A}$, where $\lambda$ is a real number representing the vorticity strength and $I_A$ denotes the characteristic function of some measurable set $A$, that is, $I_A=\lambda$ in $A$ and $I_A=0$ elsewhere. Physically, the fluid with vorticity $\omega=\lambda I_A$ is irrotational outside $A$(or there is no relative rotation between two nearby fluid particles), and rotates with constant strength in $A$. Vortex patch solutions are a typical class of non-smooth solutions of the incompressible Euler equations(including the non-stationary case) which are used to model some natural phenomena with discontinuity in vorticity, such as the hurricane.

There are mainly two kinds of steady Euler flows that are very important both physically and mathematically in the literature. The first one is of concentration type, or desingularization type, which is closely related to the point vortex model. According to the point vortex model, the evolution of $N$ concentrated vortices(i.e., the vorticity is a sum of $N$ Dirac delta measures) is governed by a Hamiltonian system of $N$ degrees of freedom with the Kirchhoff-Routh function as the Hamiltonian. See \cite{L}\cite{MP2}\cite{MP3}\cite{T3} for example. Desingularization of point vortices is to construct a family of steady Euler flows in which the vorticity is sufficiently supported near a given critical point of the Kirchhoff-Routh function.  Desingularization is a very useful tool to explore dynamically possible equilibria of planar Euler flows and analyze their stability. See \cite{CLW}\cite{CPY}\cite{CPY2}\cite{CW0}\cite{CW}\cite{CW2}\cite{SV}\cite{W} for example.
The second kind of steady Euler flows is of perturbation type. It aims to construct a new steady solution of the Euler equations near a given one. Related results can be found in \cite{CW3}\cite{LP}\cite{LYY}.
In \cite{CW3}, the authors proved that for any harmonic function $q$ corresponding to a nontrivial irrotational Euler flow, there exists a family of steady vortex patch solutions that are supported near a given finite collection of strict extreme points of $q$. In \cite{LP}\cite{LYY}, the authors also obtained similar results for smooth Euler flows. In contrast to the concentration type where the circulation has a positive lower bound, the perturbation results in \cite{CW3}\cite{LP}\cite{LYY} dealt with the case in which the circulation vanishes as the parameter changes. For the three-dimensional Euler equations, corresponding perturbation results can be found in \cite{Al}\cite{TX}, while to our knowledge there is no existence result of concentration type in the literature.

In this paper, we continue the study on the existence of steady Euler flows of perturbation type. For simplicity, we only consider the impermeability boundary condition, that is, $g\equiv0$ in \eqref{100}. We focus our attention on vortex patch solution in which the vorticity vanishes in a finite number of disjoint regions of small diameter, which we call ``dead cores"(the velocity of the fluid vanishes on these ``dead cores"), while outside these regions the vorticity is a positive constant. More precisely, we consider steady Euler flows in which the vorticity has the form
\begin{equation}\label{vvv}\omega=1-I_{\cup_{i=1}^k A_i},\end{equation}
 where each ``dead core" $A_i$ is located near some given point $x_i$.
First we show that if such a flow exists, then each $x_i$ must a critical point of the function $\psi_0$, the solution of the following elliptic problem
 \begin{equation}\label{psi00}
\begin{cases}
  -\Delta \psi_0=1, & x\in D, \\
  \psi_0=0, & x\in\partial D.
\end{cases}
\end{equation}
Then we show that given $k$ isolated maximum points of $\psi_0$, say $x_1,\cdot\cdot\cdot,x_k$, there exists a steady Euler flow of the form \eqref{vvv} where each ``dead core" $A_i$ is located near $x_i$ and $|A_i|$ is very small. In contrast to \cite{CW3}\cite{LP}\cite{LYY} where the location of the support of the vorticity is influenced by the background irrotational flow, here the location of each ``dead core" is determined by $\psi_0$. Note that our method also applies to the case $g\neq0,$ though in that situation the location of each ``dead core" is determined by $\psi_0$ and $g$ jointly.

 The main idea of the proof comes from \cite{T}. In \cite{T}, Turkington considered the maximization of the kinetic energy of the fluid
 \[E(\omega)=\frac{1}{2}\int_D|\mathbf{v}|^2dx=\frac{1}{2}\int_D\int_D|K\omega|^2dx\]
 over the following admissible class
 \[\mathcal{M}^\lambda(D)=\{\omega\in L^\infty(D)\mid0\leq\omega\leq\lambda \mbox{ a.e.},\int_D\omega(x)dx=1\},\]
 where $\lambda>0$ is a large positive number such that $\mathcal{M}^\lambda(D)$ is not empty. Turkington proved that $E$ attains its maximum over $\mathcal{M}^\lambda(D)$ and each maximizer is a steady solution of the Euler equations having the form
 $\omega^\lambda=\lambda I_{A^\lambda},$
 where $A^\lambda\subset D$ depends on $\lambda$ and is called the ``vortex core". Moreover, Turkington analyzed the asymptotic behavior of $\omega^\lambda$ and showed that $A^\lambda$ ``shrinks" to a global minimum point of the Robin function of the domain as $\lambda\to+\infty$. Although Turkington's result is of concentration type, but his idea can be generalized to construct other types of steady Euler flows. See \cite{B}\cite{B2}\cite{CW3} for example. To prove our result, following Turkinton's idea, we define a new admissible class with suitable restriction on the support of the vorticity(see \eqref{2000} in Section 3) and consider the minimization of $E$ over it. But in contrast to \cite{T}, the minimizer may not be a steady solution of the Euler equations due to the support restriction. To this end, we need to estimate the Lagrange multiplier arising in the minimization problem. This key ingredient is somewhat different from \cite{T}, and the technique here is mostly inspired by \cite{CW3}. Once we have obtained the estimate of the Lagrange multiplier, we are able to show that each ``dead core" $A_i$ ``shrinks" to $x_i$ respectively. Finally following the idea in \cite{EM} we can prove that the minimizer is indeed a steady solution.

Note that except for the variational method for the vorticity, there is another very effective way to study steady Euler flows, that is, to study the following semilinear elliptic problem satisfied by the stream function
\begin{equation}\label{semi}
\begin{cases}
-\Delta\psi=f(\psi),&x\in D,\\
\frac{\partial\psi}{\partial \mathbf{n}^\perp}=-g,&x\in\partial D,
\end{cases}
\end{equation}
where $\mathbf{n}$ denotes clockwise rotation through ${\pi}/{2}$ of $\mathbf{n}$. It can be proved that if $f:\mathbb R\to\mathbb R$ is a local Lipschitz or monotone function, then $\mathbf{v}=(\partial_{x_2} \psi,-\partial_{x_1}\psi)$ is a solution of \eqref{100}. See \cite{CW4} for example. Such a method is called the stream function method and related references are \cite{CLW}\cite{CPY}\cite{LP}\cite{LYY}\cite{SV}.

 This paper is organized as follows. In Section 2, we give the mathematical formulation of the problem and state the main results. In Section 3 and Section 4, we prove the main results.

\section{Main Results}

To begin with, we give some notations for clarity. Throughout this paper, we assume that $D\subset\mathbb R^2$ is a bounded and simply connected domain with a smooth boundary $\partial D$, and $\mathbf{n}$ denotes the outward unit normal of $\partial D$.
Let $\psi_0$ be the solution of the following elliptic problem
\begin{equation}\label{psi0}
\begin{cases}
  -\Delta \psi_0=1, & x\in D, \\
  \psi_0=0, & x\in\partial D.
\end{cases}
\end{equation}
Note that by Hopf's lemma we have $\frac{\partial \psi_0}{\partial \mathbf{n}}(x)<0$ for each $x\in\partial D.$
Let $\alpha$ be a given positive number. Define the $\alpha$-uniform cone as follows
\[\mathbb{K}^{\alpha}:=\{\vec{\kappa}\in \mathbb R^k\mid\vec{\kappa}=(\kappa_1,\cdot\cdot\cdot,\kappa_k), \kappa_1,\cdot\cdot\cdot,\kappa_k>0, \max\left\{{\kappa_i}/{\kappa_j}\mid i,j=1,\cdot\cdot\cdot,k\right\}\leq\alpha_0\}.\]
The norm of $\vec{\kappa}=(\kappa_1,\cdot\cdot\cdot,\kappa_k)$ is given by
\[|\vec{\kappa}|:=\sum_{i=1}^k\kappa_i.\]

For any measurable function $f:D\to\mathbb R,$ $supp(f)$ denotes the essential support of $f$, the complement of the union of all open sets in which $f$ vanishes.
see \cite{LL}, \S 1.5 for example. For a planar vector $\mathbf{b}=(b_1,b_2)$, the symbol $\mathbf{b}^\perp$ denotes clockwise rotation through $\pi/2$ of $\mathbf{b}$, that is, $\mathbf{b}^\perp=(b_2,-b_1)$. We also write $\nabla^\perp f=(\nabla f)^\perp$ for brevity. For a measurable set $A\subset D,$ $|A|$ denotes the two-dimensional Lebesgue measure of $A$.

Now we consider the following incompressible Euler equations in $D$ with impermeability boundary condition
\begin{equation}\label{233}
\begin{cases}
(\mathbf{v}\cdot\nabla)\mathbf{v}=-\nabla P,&x\in D,\\
 \nabla\cdot\mathbf{v}=0,&x\in D,\\
 \mathbf{v}\cdot\mathbf{n}=0,&x\in\partial D.
\end{cases}
\end{equation}
By the incompressibility condition $\nabla\cdot\mathbf{v}=0$ and the fact that $D$ is a simply connected domain, there exists a function $\psi$, called the stream function, such that $\mathbf{v}=\nabla^\perp \psi.$ This can be verified by using the Green's theorem. Taking into account the impermeability boundary condition $ \mathbf{v}\cdot\mathbf{n}=0$ on $\partial D,$ we deduce that $\psi$ is a constant on $\partial D.$ Without loss of generality, by adding a suitable constant we can assume that $\psi=0$ on $\partial D.$ Now it is easy to check that
\begin{equation}\label{234}
\begin{cases}
-\Delta\psi=\omega,&x\in D,\\
 \psi=0,&x\in \partial D.
\end{cases}
\end{equation}
Let $G$ be the Green's function for $-\Delta$ in $D$ with zero Dirichlet data. Then $\psi$ can be expressed in terms of $\omega$ as follows
\[\psi(x)=G\omega(x):=\int_DG(x,y)\omega(y)dy.\]
On the other hand, we take the curl on both sides of the first equation in \eqref{233} to obtain
\[\mathbf{v}\cdot\nabla\omega=0.\]
Therefore, the Euler system \eqref{233} can be simplified as the following vorticity equation
\begin{equation}\label{235}
\nabla^\perp G\omega\cdot\nabla\omega=0.
\end{equation}

By integration by parts, we can give the weak formulation of the vorticity equation \eqref{235}.
\begin{definition}\label{defff}
Let $\omega\in L^\infty(D)$. If $\omega$ satisfies
\begin{equation}\label{ws}\int_D\omega(x)\nabla^\perp G\omega(x)\cdot\nabla\phi(x)dx=0\end{equation}
for any test function $\phi\in C_c^\infty(D),$
then $\omega$ is called a weak solution of the vorticity equation.
\end{definition}
\begin{remark}\label{remark}
In Definition \ref{defff}, by density argument, it is easy to check that the test function $\phi$ can be chosen arbitrarily in $H^1_0(D)$.
\end{remark}

Now we are ready to state the main results of this paper. The first result is as follows.
\begin{theorem}\label{thm1}
Let $x_1,\cdot\cdot\cdot,x_k\in\overline{D}$ be $k$ different points. Suppose that for each $\vec{\kappa}\in\mathbb{K}^{\alpha}$ with $|\vec{\kappa}|<<1,$ there exists a steady vortex patch solution $\omega^{\vec{\kappa}}$ of the vorticity equation with the form
\begin{equation}\label{hole}
\omega^{\vec{\kappa}}=1-I_{\cup_{i=1}^k A^{\vec{\kappa}}_i}, |A_i^{\vec{\kappa}}|=\kappa_i, A_i^{\vec{\kappa}}\subset B_{o(1)}(x_i),
\end{equation}
where $o(1)\to0$ as $|\vec{\kappa}|\to0.$ Then each $x_i\in D$ and must be a critical point of $\psi_0.$

\end{theorem}

A natural question is the existence of steady solution of the form \eqref{hole}.
By strong maximum principle, $\psi_0$ attains its minimum value only on $\partial D,$ so there are only two types of critical points for $\psi_0$, that is, maximum points and saddle points.  The second result of this paper is concerned with the case when each $A_i^{\vec{\kappa}}$ is near an isolated maximum point of $\psi_0$.
\begin{theorem}\label{thm2}
Let $x_1,\cdot\cdot\cdot,x_k$ be $k$ different isolated maximum points of $\psi_0$. Then there exists $\delta_0>0$, such that for each $\vec{\kappa}\in\mathbb K^{\alpha},$ $|\vec{\kappa}|<\delta_0,$ there is a steady vortex patch solution of the vorticity equation with the form
\[\omega^{\vec{\kappa}}=1-I_{\cup_{i=1}^k A^{\vec{\kappa}}_i}, |A_i^{\vec{\kappa}}|=\kappa_i, A^{\vec{\kappa}}_i=\{x\in D\mid G\omega^{\vec{\kappa}}(x)= \mu^{\vec{\kappa}}_i\}\cap B_{r_0}(x_i),\]
where $\mu^{\vec{\kappa}}_i$ is a real positive number depending on $\vec{\kappa}$ and $r_0$ is a small positive number not depending on $\vec{\kappa}$. The velocity field
$\mathbf{v}=\nabla^\perp G\omega^{\vec{\kappa}}=0$ a.e. on each $A^{\vec{\kappa}}_i$.
Moreover, as $|\vec{\kappa}|\to0,$ each $A_i^{\vec{\kappa}}$ shrinks to $x_i$, that is, $A_i^{\vec{\kappa}}\subset B_{o(1)}(x_i)$.

\end{theorem}

\begin{remark}
Existence of isolated maximum points of $\psi_0$ is determined by the geometry of the domain $D$. In \cite{ML}, Makar-Limanov studied equation \eqref{psi0} and showed that if $D$ is bounded convex domain, then $\psi_0^{1/2}$ is strictly concave, therefore $\psi_0$ has a unique maximum point in $D$.
\end{remark}

\begin{remark}
It is easy to see that the corresponding stream function $\psi^{\vec{\kappa}}=G\omega^{\vec{\kappa}}$ satisfies the following elliptic problem with Heaviside nonlinearity
\begin{equation}\label{psi0}
\begin{cases}
  -\Delta \psi^{\vec{\kappa}}=1-I_{\cup_{i=1}^k \{x\in D\mid \psi^{\vec{\kappa}}= \mu^{\vec{\kappa}}_i\}\cap B_{r_0}(x_i)}, & x\in D, \\
  \psi^{\vec{\kappa}}=0, & x\in\partial D.
\end{cases}
\end{equation}

\end{remark}

\section{Proof of Theorem \ref{thm1}}

In this section, we give the proof of Theorem \ref{thm1}. The key point is to choose a suitable test function in \eqref{ws}.

\begin{proof}[Proof of Theorem \ref{thm1}]
First we exclude the possibility $x_i\in\partial D$ by contradiction. Suppose that $x_i\in \partial D$ for some index $i$. Denote $\nabla^\perp\psi_0(x_i)=\mathbf{b}$ for simplicity, then $|\mathbf{b}|>0$ by Hopf's lemma.

Since $\omega^{\vec{\kappa}}$ is a steady solution of the vorticity equation, we have for any $\phi\in C_c^\infty(D)$
\begin{equation}\label{3-1}
\int_D\omega^{\vec{\kappa}}(x)\nabla^\perp G\omega^{\vec{\kappa}}(x)\cdot\nabla\phi(x) dx=0.
\end{equation}
Now we choose $\phi$ as follows
\[\phi(x)=\chi(x)G\omega^{\vec{\kappa}}(x)\mathbf{b}\cdot x,\]
where $\chi$ is a smooth function satisfying $\chi\equiv 1$ near $x_i$ and $\chi\equiv 0$ near $x_j$ if $j\neq i$. Such $\chi$ can be constructed by using mollification technique. It is not hard to check that the $\phi$ defined above belongs to $H^1_0(D)$, so by Remark \ref{remark} it can be chosen as the test function.
Therefore we have
\begin{equation}\label{3-112}
\int_D\omega^{\vec{\kappa}}(x)\nabla^\perp G\omega^{\vec{\kappa}}(x)\cdot(\nabla \chi(x)G\omega^{\vec{\kappa}}(x)\mathbf{b}\cdot x+\chi(x)\nabla G\omega^{\vec{\kappa}}(x)\mathbf{b}\cdot x+\chi(x)G\omega^{\vec{\kappa}}(x)\mathbf{b}) dx=0.
\end{equation}
By the choice of $\chi$, we obtain
\begin{equation}\label{3-2}
\int_{A^{\vec{\kappa}}_i}\nabla^\perp G\omega^{\vec{\kappa}}(x)\cdot(\mathbf{b}G\omega^{\vec{\kappa}}(x)+\nabla G\omega^{\vec{\kappa}}(x)\mathbf{b}\cdot x) dx=0.
\end{equation}
Taking into account the fact $\nabla^\perp G\omega^{\vec{\kappa}}\cdot \nabla G\omega^{\vec{\kappa}}\equiv0$, we have
\begin{equation}\label{3-3}
\int_{A^{\vec{\kappa}}_i}\nabla^\perp G\omega^{\vec{\kappa}}(x)\cdot\mathbf{b}G\omega^{\vec{\kappa}}(x)dx=0.
\end{equation}
On the other hand, since $G\omega^{\vec{\kappa}}\to\psi_0$ in $C^1(\overline{D})$ by $L^p$ estimate and Sobolev embedding and $A^{\vec{\kappa}}_i\subset B_{o(1)}(x_i)$ as $|\vec{\kappa}|\to 0$, we have
\begin{equation}
\lim_{|\vec{\kappa}|\to 0}\sup_{x\in A^{\vec{\kappa}}_i}|\nabla^\perp G\omega^{\vec{\kappa}}(x)-\mathbf{b}|\to 0,
\end{equation}
 from which we deduce that
 \begin{equation}\label{bbbb}
 |\nabla^\perp G\omega^{\vec{\kappa}}(x)\cdot\mathbf{b}|\geq\frac{|\mathbf{b}|^2}{2},\,\,\forall\,x\in A^{\vec{\kappa}}_i
 \end{equation}
  if $|\vec{\kappa}|$ is sufficiently small.
Combining \eqref{3-3} and \eqref{bbbb} we get the following contradiction
 \begin{equation}\label{3-2}
0<\int_{A^{\vec{\kappa}}_i}\frac{|\mathbf{b}|^2}{2}G\omega^{\vec{\kappa}}dx\leq\int_{A^{\vec{\kappa}}_i}\nabla^\perp G\omega^{\vec{\kappa}}\cdot\mathbf{b}G\omega^{\vec{\kappa}}dx=0,
\end{equation}
where we used the fact that $G\omega^{\vec{\kappa}}>0$ in $D$ by strong maximum principle. Therefore we have proved that each $x_i$ must be an interior point of $D$.

Next we show that each $x_i$ must be a critical point of $\psi_0$. Denote $\nabla^\perp\psi_0(x_i)=\mathbf{c}_i.$ It suffices to prove $\mathbf{c}_i=0.$ To this end, we choose the test function in \eqref{3-1} as
\[\phi(x)=\vartheta(x)\mathbf{c}_i\cdot x,\]
 where $\vartheta\in C_c^\infty(D)$ satisfying $\vartheta\equiv 1 $ near $x_i$ and $\vartheta\equiv0 $ near $x_j$ for $j\neq i$.
Then we have
\[\int_{A^{\vec{\kappa}}_i}\nabla^\perp G\omega^{\vec{\kappa}}(x)\cdot\mathbf{c}_i dx=0,\]
or
\begin{equation}\label{3-4}
\frac{1}{\kappa_i}\int_{A^{\vec{\kappa}}_i}\nabla^\perp G\omega^{\vec{\kappa}}(x)\cdot\mathbf{c}_i dx=0.
\end{equation}
Since $G\omega^{\vec{\kappa}}\to \psi_0$ in $C^1(\overline{D})$ and $A^{\vec{\kappa}}_i\subset B_{o(1)}(x_i)$ as $|\vec{\kappa}|\to 0$, we have $G\omega^{\vec{\kappa}}\to\mathbf{c}_i$ uniformly on $A^{\vec{\kappa}}_i$. Taking into account \eqref{3-4} we obtain
\[0=\frac{1}{\kappa_i}\int_{A^{\vec{\kappa}}_i}\nabla^\perp G\omega^{\vec{\kappa}}(x)\cdot\mathbf{c}_i dx\to|\mathbf{c}_i|^2\]
as $|\vec{\kappa}|\to0,$ which verifies $\mathbf{c}_i=0.$

\end{proof}

\section{Proof of Theorem \ref{thm2}}
In this section, we give the proof of Theorem \ref{thm2}. As mentioned in Section 1, the basic idea is to solve a minimization problem for the kinetic energy of the fluid subject to some appropriate constraints for the vorticity and analyze the limiting behavior.

 For an ideal fluid with impermeability boundary condition, the kinetic energy is
 \[E(\omega)=\frac{1}{2}\int_D|\nabla^\perp G\omega(x)|^2dx=\frac{1}{2}\int_D|\nabla G\omega(x)|^2dx=\frac{1}{2}\int_D\int_DG(x,y)\omega(x)\omega(y)dxdy.\]
We consider the minimization of $E$ over the following admissible class
\begin{equation}\label{2000}\mathcal{K}^{\vec{\kappa}}(D)=\{\omega=1-w\mid w=\sum_{i=1}^kw_i,0\leq w_i\leq 1,\int_Dw_i(x)dx=\kappa_i,supp(w_i)\subset B_{r_0}(x_i)\},\end{equation}
where $r_0$ is chosen to be sufficiently small such that $x_i$ is the unique maximum point of $\psi_0$ on $\overline{B_{r_0}(x_i)}$ for each $i=1,\cdot\cdot\cdot,k.$

In the following, for convenience we define
\[\mathcal{N}^{\vec{\kappa}}(D)=\{w\in L^\infty(D)\mid w=\sum_{i=1}^kw_i,0\leq w_i\leq 1,\int_Dw_i(x)dx=\kappa_i,supp(w_i)\subset B_{r_0}(x_i)\}.\]
Then it is obvious that to minimize $E$ over $\mathcal{K}^{\vec{\kappa}}(D)$ is equivalent to minimize $\mathcal{E}$ over $\mathcal{N}^{\vec{\kappa}}(D),$ where
\[\mathcal{E}(w):=\frac{1}{2}\int_D\int_DG(x,y)w(x)w(y)dxdy-\int_D\psi_0(x)w(x)dx.\]

\subsection{Existence} First we study the existence of a minimizer and its profile.
\begin{lemma}
There exists a unique minimizer of $E$ over $\mathcal{K}^{\vec{\kappa}}(D).$
\end{lemma}
\begin{proof}
As mentioned above, it suffices to prove that $\mathcal{E}$ attains its minimum over $\mathcal{N}^{\vec{\kappa}}(D)$.
Since $\mathcal{N}^{\vec{\kappa}}(D)$ is a sequentially compact subset of $L^\infty(D)$ in the weak star topology(see \cite{CW3}, Lemma 3.1 for example) and $\mathcal{E}$ is weakly star continuous(see \cite{CW3}, Lemma 3.2), the result can be easily obtained by the direct method of the calculus of variations.

As to uniqueness, we need only to use the fact that $\mathcal{N}^{\vec{\kappa}}(D)$ is a convex subset in $L^\infty(D)$ and $\mathcal{E}$ is a strict convex functional over $\mathcal{N}^{\vec{\kappa}}(D)$. In fact, for any $w_1,w_2\in\mathcal{N}^{\vec{\kappa}}(D)$ with $w_1\neq w_2$ and $\theta\in(0,1)$, we have
\begin{equation}\label{con}
\begin{split}
\mathcal{E}(\theta w_1+(1-\theta)w_2)=&E(\theta w_1+(1-\theta)w_2)-\int_D\psi_0(x)(\theta w_1(x)+(1-\theta)w_2(x))dx\\
=&\theta^2E(w_1)+(1-\theta)^2E(w_2)+\theta(1-\theta)\int_D\int_DG(x,y)w_1(x)w_2(y)dxdy\\
&-\int_D\psi_0(x)(\theta w_1(x)+(1-\theta)w_2(x))dx\\
=&\theta^2E(w_1)+(1-\theta)^2E(w_2)+{\theta(1-\theta)}(E(w_1)+E(w_2)-E(w_1-w_2)\\
&-\int_D\psi_0(x)(\theta w_1(x)+(1-\theta)w_2(x))dx\\
=&\theta E(w_1)+(1-\theta) E(w_2)-\int_D\psi_0(x)(\theta w_1(x)+(1-\theta)w_2(x))dx\\&-\theta(1-\theta)E(w_1-w_2)\\
=&\theta\mathcal{E}(w_1)+(1-\theta) \mathcal{E}(w_2)-\theta(1-\theta)E(w_1-w_2)\\
>&\theta\mathcal{E}(w_1)+(1-\theta) \mathcal{E}(w_2),
\end{split}
\end{equation}
where we used the symmetry of the Green's function
and \[\int_D\int_DG(x,y)(w_1(x)-w_2(x))(w_1(y)-w_2(y))dxdy=\int_D|\nabla (Gw_1-Gw_2)(x)|^2dx\geq0\]
by integration by parts.
\end{proof}

\begin{lemma}
Let $\omega^{\vec{\kappa}}$ be the minimizer of $E$ over $\mathcal{K}^{\vec{\kappa}}(D).$ Then $\omega^{\vec{\kappa}}$ has the form
\[\omega^{\vec{\kappa}}=1-\sum_{i=1}^kI_{\{G\omega^{\vec{\kappa}}\geq\mu^{\vec{\kappa}}_i\}\cap B_{r_0}(x_i)},\]
where each $\mu^{\vec{\kappa}}_i$ is a real number depending on ${\vec{\kappa}}.$
\end{lemma}
\begin{proof}
Notice that $w^{\vec{\kappa}}=1-\omega^{\vec{\kappa}}$ is a minimizer of $\mathcal{E}$ over $\mathcal{N}^{\vec{\kappa}}(D).$
For fixed $i$, we define a family of test functions $w_s^{\vec{\kappa}}=w^{\vec{\kappa}}+s(z_0-z_1),$ where $s>0$ is a parameter, $z_0,z_1$ satisfy
\begin{equation}
\begin{cases}
z_0,z_1\in L^\infty(D),\,z_0,z_1\geq 0,\text{ a.e. in } D,

 \\ \int_Dz_0(x)dx=\int_D z_1(x)dx,
 \\supp(z_0), supp(z_1)\subset B_{r_0}(x_i),
 \\z_0=0 \quad\text{in } D\setminus\{x\in D\mid w^{\vec{\kappa}}(x)\leq 1-\delta\},
 \\z_1=0 \quad\text{in } D\setminus\{x\in D\mid w^{\vec{\kappa}}(x)\geq\delta\}.
\end{cases}
\end{equation}
Here $\delta$ is a positive number. It is not difficult to check that for fixed $z_0,z_1$ and $\delta$, $w^{\vec{\kappa}}_s\in \mathcal{N}^{\vec{\kappa}}(D)$ provided that $s$ is sufficiently small. Since $w^{\vec{\kappa}}$ is a minimizer, we have
\[0\leq\frac{d\mathcal{E}(w_s^{\vec{\kappa}})}{ds}|_{s=0^+}=\int_Dz_0(x)(Gw^{\vec{\kappa}}(x)-\psi_0(x))dx
-\int_Dz_1(x)(Gw^{\vec{\kappa}}(x)-\psi_0(x))dx.\]
By the choice of $z_0$ and $z_1$ we obtain
\begin{equation}\label{3-100}
\sup_{\{x\in D\mid w^{\vec{\kappa}}(x)>0\}\cap{B_{r_0}(x_i)}}(Gw^{\vec{\kappa}}-\psi_0)\leq\inf_{\{x\in D\mid w^{\vec{\kappa}}(x)<1\}\cap{B_{r_0}(x_i)}}(Gw^{\vec{\kappa}}-\psi_0).
\end{equation}
Since $\overline{\{x\in D\mid w^{\vec{\kappa}}(x)>0\}\cap{B_{r_0}(x_i)}}\cap\overline{\{x\in D\mid w^{\vec{\kappa}}(x)<1\}\cap{B_{r_0}(x_i)}}\neq\varnothing$ and $Gw^{\vec{\kappa}}-\psi_0$ is continuous, \eqref{3-100} is in fact an equality, that is,
\begin{equation}
\sup_{\{x\in D\mid w^{\vec{\kappa}}(x)>0\}\cap{B_{r_0}(x_i)}}(Gw^{\vec{\kappa}}-\psi_0)=\inf_{\{x\in D\mid w^{\vec{\kappa}}(x)<1\}\cap{B_{r_0}(x_i)}}(Gw^{\vec{\kappa}}-\psi_0):=-\mu^{\vec{\kappa}}_i.
\end{equation}
Now it is easy to check that
\begin{equation}
\begin{cases}
w^{\vec{\kappa}}=0\text{\,\,\,\,\,\,a.e.\,} \text{on }\{x\in D\mid Gw^{\vec{\kappa}}(x)-\psi_0(x)>-\mu^{\vec{\kappa}}_i\}\cap B_{r_0}(x_i),
 \\ w^{\vec{\kappa}}=1\text{\,\,\,\,\,\,a.e.\,} \text{on }\{x\in D\mid Gw^{\vec{\kappa}}(x)-\psi_0(x)<-\mu^{\vec{\kappa}}_i\}\cap B_{r_0}(x_i).
\end{cases}
\end{equation}
On the level set $\{x\in D\mid Gw^{\vec{\kappa}}(x)-\psi_0(x)=-\mu^{\vec{\kappa}}_i\}\cap B_{r_0}(x_i)$,
by the property of Sobolev functions, we have $\nabla (Gw^{\vec{\kappa}}-\psi_0)=0\text{\,\,a.e.}$, therefore $w^{\vec{\kappa}}=-\Delta G\omega^\lambda=-\Delta\psi_0=1\text{\,\,a.e.}$.
Altogether, we obtain
\[w^{\vec{\kappa}}=\sum_{i=1}^k I_{\{x\in D\mid Gw^{\vec{\kappa}}(x)-\psi_0(x)\leq-\mu_i^{\vec{\kappa}}\}\cap B_{r_0}(x_i)}.\]
from which we deduce that
\[\omega^{\vec{\kappa}}=1-\sum_{i=1}^kI_{\{G\omega^{\vec{\kappa}}\geq \mu^{\vec{\kappa}}_i\}\cap B_{r_0}(x_i)}.\]

\end{proof}

From now on, for simplicity we denote $A^{\vec{\kappa}}_i=\{x\in D\mid G\omega^{\vec{\kappa}}(x)\geq\mu_i^{\vec{\kappa}}\}\cap B_{r_0}(x_i),$ which is called the ``dead core". We will see in the following that the velocity of the fluid vanishes on $A^{\vec{\kappa}}_i$ if $|\vec{\kappa}|$ is sufficiently small.

\subsection{Limiting behavior}
As in \cite{EM}, to show that the minimizer $\omega^{\vec{\kappa}}$ is a weak solution of the vorticity equation \eqref{235}, we need to show that each $A^{\vec{\kappa}}_i$ is away from $\partial B_{r_0}(x_i)$.
To this end, we need to analyze the limiting behavior of the minimizer as $|\vec{\kappa}|\to0.$

We begin with the following estimate of the upper bound of the energy. For simplicity, we will use $C$ to denote various positive numbers not depending on $\vec{\kappa}$, but possibly depending on $D$, $\alpha, r_0$ and $x_1,\cdot\cdot\cdot,x_k$.
\begin{lemma}\label{1-1-2}
$\mathcal{E}(w^{\vec{\kappa}})\leq-\sum_{i=1}^k\psi_0(x_i)\kappa_i+C|\vec{\kappa}|^{3/2}.$
\end{lemma}
\begin{proof}
 We define a family of test functions $v^{\vec{\kappa}}=\sum_{i=1}^kI_{B_{\varepsilon_i}(x_i)}$, where $\varepsilon_i=\sqrt{\kappa_i/\pi}$.  It is obvious that $v^{\vec{\kappa}}\in\mathcal{N}^{\vec{\kappa}}(D)$ if $|{\vec{\kappa}}|$ is sufficiently small. Therefore we have $\mathcal{E}(w^{\vec{\kappa}})\leq\mathcal{E}(v^{\vec{\kappa}})$.

Now we calculate $\mathcal{E}(v^{\vec{\kappa}})$ as follows. Recall that
\begin{equation}\label{3-99}
\mathcal{E}(v^{\vec{\kappa}})=\frac{1}{2}\int_DGv^{\vec{\kappa}}(x)v^{\vec{\kappa}}(x)dx
-\int_D\psi_0(x)v^{\vec{\kappa}}(x)dx.\end{equation}
For the first term, by Sobolev embedding theorem and $L^p$ estimate we have
\begin{equation}\label{3-101}
\begin{split}
\left|\frac{1}{2}\int_DGv^{\vec{\kappa}}(x)v^{\vec{\kappa}}(x)dx\right|
\leq \frac{1}{2}|Gv^{\vec{\kappa}}|_{L^\infty(D)}|\vec{\kappa}|
\leq C|Gv^{\vec{\kappa}}|_{W^{2,2}(D)}|\vec{\kappa}|
\leq C|v^{\vec{\kappa}}|_{L^2(D)}|\vec{\kappa}|
\leq C|\vec{\kappa}|^{3/2}.
\end{split}
\end{equation}
For the second term,
\begin{equation}\label{3-102}
\begin{split}
\int_D\psi_0(x)v^{\vec{\kappa}}(x)dx=&\sum_{i=1}^k\int_{B_{\varepsilon}(x_i)}\psi_0(x)dx\\
=&\sum_{i=1}^k\int_{B_{\varepsilon_i}(x_i)}(\psi_0(x)-\psi_0(x_i))dx+\sum_{i=1}^k\int_{B_{\varepsilon_i}(x_i)}\psi_0(x_i)dx\\
=&\sum_{i=1}^k\int_{B_{\varepsilon_i}(x_i)}(\psi_0(x)-\psi_0(x_i))dx+\sum_{i=1}^k\kappa_i\psi_0(x_i).
\end{split}
\end{equation}
Since $\psi_0\in C^1(\overline{D})$, we have
\begin{equation}\label{3-103}
\begin{split}
\left|\int_{B_{\varepsilon_i}(x_i)}(\psi_0(x)-\psi_0(x_i))dx\right|&\leq \int_{B_{\varepsilon_i}(x_i)}\left|\psi_0(x)-\psi_0(x_i)\right|dx\\
&\leq \int_{B_{\varepsilon_i}(x_i)}|\nabla\psi_0|_{L^\infty(D)}|x-x_i|dx\\
&\leq \varepsilon_i |\nabla\psi_0|_{L^\infty(D)}\kappa_i\\
&\leq C|\vec{\kappa}|^{3/2}.
\end{split}
\end{equation}
The desired result follows from \eqref{3-99},\eqref{3-101},\eqref{3-102} and \eqref{3-103}.

\end{proof}
 As in \cite{T}, we also need the lower bound of the kinetic energy of each ``dead core" $A_i^{\vec{\kappa}}$.
\begin{lemma}\label{1-1-1}
For each fixed index $i$, we have
\[\int_D(Gw^{\vec{\kappa}}(x)-\psi_0(x)+\mu^{\vec{\kappa}}_i)w^{\vec{\kappa}}_i(x)dx\geq-C{\kappa}_i^{{3}/{2}}.\]
\end{lemma}
\begin{proof}
Recall that $w_i^{\vec{\kappa}}=I_{A_i^{\vec{\kappa}}}$, where $A^{\vec{\kappa}}_i=\{x\in D\mid Gw^{\vec{\kappa}}(x)-\psi_0(x)+\mu^{\vec{\kappa}}_i\leq0\}\cap B_{r_0}(x_i)$. Denote $\zeta^{\vec{\kappa}}=Gw^{\vec{\kappa}}-\psi_0+\mu^{\vec{\kappa}}_i$. By H\"older's inequality, we have
\begin{equation}\label{3-201}
\int_D\zeta^{\vec{\kappa}}(x)w_i^{\vec{\kappa}}(x)dx=\int_{A^{\vec{\kappa}}_i}\zeta^{\vec{\kappa}}(x)dx\geq
-|A^{\vec{\kappa}}_i|^{1/2}\left(\int_{A^{\vec{\kappa}}_i}|\zeta^{\vec{\kappa}}(x)|^2dx\right)^{1/2}
=-\kappa_i^{1/2}|\zeta^{\vec{\kappa}}|_{L^2(A^{\vec{\kappa}}_i)}.
\end{equation}
On the other hand, by Sobolev embedding $W^{1,1}(B_{r_0}(x_i))\hookrightarrow L^2(B_{r_0}(x_i))$, we have
\begin{equation}\label{3-202}
\begin{split}
|\zeta^{\vec{\kappa}}|_{L^2(A^{\vec{\kappa}}_i)}= &|\zeta_{-}^{\vec{\kappa}}|_{L^2(B_{r_0}(x_i))}\\
\leq & C\left(\int_{B_{r_0}(x_i)}|\zeta_{-}^{\vec{\kappa}}(x)|dx+\int_{B_{r_0}(x_i)}|\nabla\zeta^{\vec{\kappa}}_{-}(x)|dx\right)\\
\leq &C\kappa^{1/2}_i |\zeta^{\vec{\kappa}}|_{L^2(A^{\vec{\kappa}}_i)}+C\kappa^{1/2}_i |\nabla\zeta^{\vec{\kappa}}|_{L^2(A^{\vec{\kappa}}_i)}.
\end{split}
\end{equation}
When $|\vec{\kappa}|$ is sufficiently small, we deduce from \eqref{3-202} that
\begin{equation}\label{3-203}
\begin{split}
|\zeta^{\vec{\kappa}}|_{L^2(A^{\vec{\kappa}}_i)}\leq C\kappa^{1/2}_i |\nabla\zeta^{\vec{\kappa}}|_{L^2(A^{\vec{\kappa}}_i)}.
\end{split}
\end{equation}
Combining \eqref{3-201} and \eqref{3-203}, we have
\begin{equation}\label{3-204}
\begin{split}
\int_D\zeta^{\vec{\kappa}}(x)w_i^{\vec{\kappa}}(x)dx&\geq-\kappa_i^{1/2}|\zeta^{\vec{\kappa}}|_{L^2(A^{\vec{\kappa}}_i)}\\
&\geq -C\kappa_i |\nabla\zeta^{\vec{\kappa}}|_{L^2(A^{\vec{\kappa}}_i)}\\
&\geq -C\kappa_i (|\nabla Gw^{\vec{\kappa}}|_{L^2(A^{\vec{\kappa}}_i)}+|\nabla\psi_0|_{L^2(A^{\vec{\kappa}}_i)})\\
&\geq -C\kappa_i^{3/2}(|\nabla Gw^{\vec{\kappa}}|_{L^\infty(A^{\vec{\kappa}}_i)}+|\nabla\psi_0|_{L^\infty(A^{\vec{\kappa}}_i)})\\
&\geq-C\kappa_i^{3/2},
\end{split}
\end{equation}
which is the desired result.
\end{proof}

Now we begin to estimate the Lagrange multiplier $\mu^{\vec{\kappa}}_i$, which is the key ingredient of the proof. This is somewhat different from \cite{T}.

 The upper bound of each $\mu^{\vec{\kappa}}_i$ is as follows.
\begin{lemma}\label{1-1-3}
$\mu^{\vec{\kappa}}_i<\psi_0(x_i).$
\end{lemma}
\begin{proof}
Since $A^{\vec{\kappa}}_i$ is obviously not empty, we can choose a point $x\in A^{\vec{\kappa}}_i$ to deduce that
\[Gw^{\vec{\kappa}}(x)-\psi_0(x)+\mu^{\vec{\kappa}}_i\leq 0,\]
or equivalently
\begin{equation}\label{3-301}
\mu^{\vec{\kappa}}_i\leq \psi_0(x)-Gw^{\vec{\kappa}}(x).
\end{equation}
Taking into account the facts that $\psi_0(x)<\psi_0(x_i)$ and $Gw^{\vec{\kappa}}>0$ in $D$(by strong maximum principle), we get the desired result.
\end{proof}
The lower bound of each $\mu^{\vec{\kappa}}_i$ is as follows.
\begin{lemma}\label{1-1-10}
For each fixed index $i$, we have
$\mu^{\vec{\kappa}}_i>\psi_0(x_i)-C|\vec{\kappa}|^{1/2}.$
\end{lemma}
\begin{proof}
We write $\mathcal{E}(w^{\vec{\kappa}})$ as follows
\begin{equation}\label{3-302}
\begin{split}
\mathcal{E}(w^{\vec{\kappa}})&=\frac{1}{2}\int_DGw^{\vec{\kappa}}(x)w^{\vec{\kappa}}(x)dx
-\int_D\psi_0(x)w^{\vec{\kappa}}(x)dx\\
&=-\frac{1}{2}\int_DGw^{\vec{\kappa}}(x)w^{\vec{\kappa}}(x)dx
+\int_D(Gw^{\vec{\kappa}}(x)-\psi_0(x))w^{\vec{\kappa}}(x)dx\\
&=-\frac{1}{2}\int_DGw^{\vec{\kappa}}(x)w^{\vec{\kappa}}(x)dx
+\sum_{j=1}^k\int_D(Gw^{\vec{\kappa}}(x)-\psi_0(x)+\mu_j^{\vec{\kappa}})w_j^{\vec{\kappa}}(x)dx
-\sum_{j=1}^k\kappa_j\mu_j^{\vec{\kappa}}.
\end{split}
\end{equation}
For the first term, by $L^p$ estimate and Sobolev embedding, we have
\begin{equation}\label{3-303}
\left|\frac{1}{2}\int_DGw^{\vec{\kappa}}(x)w^{\vec{\kappa}}(x)dx\right|\leq C|\vec{\kappa}|^{\frac{3}{2}}.
\end{equation}
For the second term, by Lemma \ref{1-1-1}, we have
\begin{equation}\label{3-304}
\sum_{j=1}^k\int_D(Gw^{\vec{\kappa}}(x)-\psi_0(x)+\mu_j^{\vec{\kappa}})w_j^{\vec{\kappa}}(x)dx\geq-C|\vec{\kappa}|^{\frac{3}{2}}.
\end{equation}
Combining Lemma \ref{1-1-2}, \eqref{3-302}, \eqref{3-302} and \eqref{3-303}, we deduce that
\begin{equation}\label{3-3-5}
\sum_{j=1}^k\kappa_j\mu_j^{\vec{\kappa}}\geq \sum_{j=1}^k\kappa_j\psi_0(x_j)-C|\vec{\kappa}|^{\frac{3}{2}}.
\end{equation}
On the other hand, by Lemma \ref{1-1-3} we deduce that
\begin{equation}\label{3-307}
\sum_{j\neq i,j=1}^k\kappa_i\mu^{\vec{\kappa}}_i\leq \sum_{j\neq i,j=1}^k\kappa_i\psi_(x_i).
\end{equation}
It follows from \eqref{3-3-5} and \eqref{3-307} that
\begin{equation}\label{3-308}
\kappa_i\mu^{\vec{\kappa}}_i\geq\kappa_i\psi_0(x_i)-C|\vec{\kappa}|^{\frac{3}{2}}.
\end{equation}
Now the desired result follows from \eqref{3-308} and the fact that $\mathbb{K}^\alpha$ is an $\alpha$-uniform cone.
\end{proof}

\begin{lemma}\label{1-1-19}
For each fixed index $i$, we have
\[\lim_{|\vec{\kappa}|\to 0}\sup_{x\in A^{\vec{\kappa}}_i}|G\omega^{\vec{\kappa}}(x)-\psi_0(x_i)|=0.\]
\end{lemma}
\begin{proof}
For each $x\in A^{\vec{\kappa}}_i$, by the definition of $A^{\vec{\kappa}}_i$
we have $G\omega^{\vec{\kappa}}(x)\geq\mu^{\vec{\kappa}}_i.$
Taking into account Lemma \ref{1-1-10}, we obtain
\begin{equation}\label{3-311}
G\omega^{\vec{\kappa}}(x) \geq\psi_0(x_i)-C|\vec{\kappa}|^{1/2}.
\end{equation}
On the other hand, by maximum principle and the fact that $x_i$ is the unique maximum point of $\psi_0$ on $\overline{B_{r_0}(x_i)}$, we have
\begin{equation}\label{3-312}
G\omega^{\vec{\kappa}}(x) \leq\psi_0(x)\leq\psi_0(x_i).
\end{equation}
Now the desired result follows from \eqref{3-311} and \eqref{3-312}.
\end{proof}

\begin{lemma}\label{1-1-1001}
$A^{\vec{\kappa}}_i\subset B_{o(1)}(x_i)$ as $|\vec{\kappa}|\to0$.
\end{lemma}

\begin{proof}
We prove this by contradiction. Suppose that there exist $\delta_1>0$, $\vec{\kappa}_n\in\mathbb K^\alpha, |\vec{\kappa}_n|<1/n$ and $x_n\in A^{\vec{\kappa}}_i$ such that $|x_n-x_i|\geq\delta_i$, where $n=1,2,3,\cdot\cdot\cdot.$ Since $G\omega^{\vec{\kappa}}\to\psi_0$ in $C^1(\overline{D})$ and $x_i$ is the unique maximum point of $\psi_0$ on $\overline{B_{r_0}(x_i)}$, we have
\[\inf_{n}|G\omega^{\vec{\kappa}}(x_n)-\psi_0(x_i)|>0,\]
which contradicts Lemma \ref{1-1-19}.
\end{proof}

\begin{lemma}
$A^{\vec{\kappa}}_i=\{x\in D\mid G\omega^{\vec{\kappa}}(x)=\mu^{\vec{\kappa}}_i\}\cap B_{r_0}(x_i)$  and $\mathbf{v}=\nabla^\perp G\omega^{\vec{\kappa}}=0$ a.e. on $A^{\vec{\kappa}}_i$ if $|\vec{\kappa}|$ is sufficiently small.
\end{lemma}
\begin{proof}
To prove $A^{\vec{\kappa}}_i=\{x\in D\mid G\omega^{\vec{\kappa}}(x)=\mu^{\vec{\kappa}}_i\}\cap B_{r_0}(x_i)$, it suffices to show that $B^{\vec{\kappa}}_i:=\{x\in D\mid G\omega^{\vec{\kappa}}(x)>\mu^{\vec{\kappa}}_i\}\cap B_{r_0}(x_i)=\varnothing.$ We prove this by contradiction. By Lemma \ref{1-1-1001}, we have $\partial{B^{\vec{\kappa}}_i}\subset \{x\in D\mid G\omega^{\vec{\kappa}}(x)=\mu^{\vec{\kappa}}_i\}\cap B_{r_0}(x_i)$, or equivalently $G\omega^{\vec{\kappa}}=\mu^{\vec{\kappa}}_i$ on $\partial B^{\vec{\kappa}}_i$ if $|\vec{\kappa}|$ is sufficiently small. But $G\omega^{\vec{\kappa}}$ is harmonic in the open set $B^{\vec{\kappa}}_i$, so by strong maximum principle we deduce that $G\omega^{\vec{\kappa}}\equiv \mu^{\vec{\kappa}}_i$ in $B^{\vec{\kappa}}_i$, which is an obvious contradiction.

To prove that $\mathbf{v}$ vanishes on $A^{\vec{\kappa}}_i$, we observe that $A^{\vec{\kappa}}_i$ is included in a level set of $G\omega^{\vec{\kappa}}$. Then by the property of Sobolev functions(see \cite{EG}, Theorem 4.4), we have $\nabla G\omega^{\vec{\kappa}}=0$ a.e. on $A^{\vec{\kappa}}_i$ as desired.
\end{proof}

\subsection{Proof of Theorem \ref{thm2}} Now we are ready to finish the proof Theorem \ref{thm2}. We need only to show that $\omega^{\vec{\kappa}}$ is a steady solution of the vorticity equation if $|\vec{\kappa}|$ is sufficiently small, since other assertions have been verified in the previous subsections.

\begin{proof}[Proof of Theorem \ref{thm2}]
It suffices to show that
\begin{equation}\label{3-1000}
\int_D\omega^{\vec{\kappa}}(x)\nabla^\perp G\omega^{\vec{\kappa}}(x)\cdot\nabla\phi(x)dx=0,\,\,\forall\,\phi\,\in C_c^\infty(D),
\end{equation}
if $|\vec{\kappa}|$ is sufficiently small.
For any $\phi\in C_c^\infty(D)$, we construct a family of smooth transformations from $D$ to $D$ by solving the following ordinary differential equation
\begin{equation}\label{4-1}
\begin{cases}\frac{d\Phi_t(x)}{dt}=\nabla^\perp\phi(\Phi_t(x)), \,\,t\in\mathbb R, \\
\Phi_0(x)=x.
\end{cases}
\end{equation}
Since $\nabla^\perp\phi$ is a smooth vector field with compact support, $\eqref{4-1}$ admits a unique solution. Furthermore, since $\nabla^\perp\phi$ is obviously a divergence-free vector field, we know from Liouville theorem(see \cite{MP2}, Appendix 1.1) that $\Phi_t$ is an area-preserving transformation from $D$ to $D$, that is, for any measurable set $A\subset D$ there holds $|\{\Phi_t(x)\mid x\in A\}|=|A|$. Let $\{\omega^{\vec{\kappa}}_t\}_{t\in\mathbb R}$ be a family of test functions defined by
\begin{equation}
\omega^{\vec{\kappa}}_t(x):=\omega^{\vec{\kappa}}(\Phi_t(x)).
\end{equation}
By Lemma \ref{1-1-1001} and the fact that $\Phi_t$ is area-preserving, it is not difficult to check that $\omega^{\vec{\kappa}}_t\in \mathcal{K}^{\vec{\kappa}}(D)$ if $|\vec{\kappa}|$ is sufficiently small, so we have
\[\frac{dE(\omega^{\vec{\kappa}}_t)}{dt}\bigg|_{t=0}=0.\]
 Expanding $E(\omega^{\vec{\kappa}}_t)$ at $t=0$, we obtain for $|t|$ small
\[\begin{split}
E(\omega^{\vec{\kappa}}_t)=&\frac{1}{2}\int_D\int_DG(x,y)\omega^{\vec{\kappa}}(\Phi_t(x))\omega^{\vec{\kappa}}(\Phi_t(y))dxdy\\
=&\frac{1}{2}\int_D\int_DG(\Phi_{-t}(x),\Phi_{-t}(y))\omega^{\vec{\kappa}}(x)\omega^{\vec{\kappa}}(y)dxdy\\
=&E(\omega^{\vec{\kappa}})+t\int_D\omega^{\vec{\kappa}}(x)\nabla^\perp G\omega^{\vec{\kappa}}(x)\cdot\nabla\phi(x)dx+o(t).
\end{split}\]
 Therefore we get
\[\int_D\omega^{\vec{\kappa}}(x)\nabla^\perp G\omega^{\vec{\kappa}}(x)\cdot\nabla\phi(x) dx=0,\,\,\forall\,\phi\,\in C_c^\infty(D).\]
This finishes the proof.
\end{proof}

\end{document}